\documentclass[12pt,a4paper]{article}
\usepackage{amsmath,euscript}
\usepackage{amssymb}
\usepackage{amsthm}
\usepackage{enumerate}
\usepackage{amsfonts}
\usepackage{comment}
\usepackage{amscd}

\newtheorem{theorem}{Theorem}[section]

\newtheorem{proposition}{Proposition}[section]
\newtheorem{lemma}[theorem]{ Lemma}%[section]
\newtheorem{corollary}[theorem]{Corollary}%[section]

\title{Growth, entropy and commutativity     
 of  algebras satisfying prescribed relations}

\author{Agata Smoktunowicz\thanks{This research was funded by ERC Advanced Grant 320974. }}

\date{  }

\begin{document}

\maketitle

\begin{abstract} In 1964, Golod and Shafarevich found that, provided that the number of relations of each degree satisfy some bounds, there exist infinitely dimensional algebras satisfying the relations. These algebras are called Golod-Shafarevich algebras. This paper provides bounds for the growth function on images of Golod-Shafarevich 
 algebras based upon the number of defining relations. This extends results from \cite{sb}, \cite{s}. 
   Lower bounds of growth for constructed algebras are also obtained, permitting the construction of algebras with various growth functions of various entropies. 
In particular, the paper answers a question by  Drensky \cite{dren} by constructing algebras with subexponential growth satisfying given relations, under mild assumption on the number of generating relations of each degree. Examples of nil algebras with neither polynomial nor exponential growth over uncountable fields are also constructed, answering a question by Zelmanov \cite{zel}. 

 Recently, several open questions concerning the commutativity of algebras satisfying a prescribed number of defining relations have arisen from the study of noncommutative singularities. Additionally, this paper solves one such question,  posed by Donovan and Wemyss in \cite{dw}.
\end{abstract}
 


\section*{Introduction}\label{S0}

  The Golod-Shafarevich theorem \cite{gs} is a beautiful result in noncommutative ring theory, which shows that  Golod-Shafarevich algebras are infinite dimensional.
 The Golod-Shafarevich theorem has been used to solve many important questions, ranging from group theory and noncommutative ring theory, to number theory. For example, Golod  used it to show that there exist infinite torsion groups which are finitely generated, to solve the General Burnside Problem \cite{g}, and, in collaboration with Shafarevich,  to solve the Class Field Problem in number theory \cite{gs}. However, despite these successes,  many inspiring open questions remain. For example, Anick asked whether the converse of the Golod-Shafarevich theorem is true, and he also asked whether the minimal possible Hilbert series of an algebra defined by a given number of relations given by the Golod-Shafarevich theorem is attainable \cite{an} (Anick  managed to solve this latter question in some special cases \cite{an}). Interesting results  related to Anick's questions were obtained by Wisliceny \cite{wis}, and more recently by Iyudu and Shakarin \cite{is}, \cite{is2}. 
 
This paper explores the  growth and entropy of algebras satisfying prescribed relations. 
Information on growth often contains important information about the properties of algebras; for example, Jategoankar showed that domains with subexponential growth satisfy the Ore condition, and hence 
 have Ore rings of fractions (\cite{kl}, pp. 48), and Zhang and Stephenson found that noetherian graded algebras  cannot have exponential growth \cite{sz} (for other interesting results see \cite{br}, \cite{gz}, \cite{kl}, \cite{rowen}, \cite{sap}).
  Assumptions about the growth of algebras are often used  in studies related to   
noncommutative (projective) algebraic geometry (see for example \cite{as}, \cite{bs}, \cite{sw}). The growth of algebras provides   information about the extended centroids of the given rings \cite{smz}, \cite{bs}.
 Combinatorial methods also find applications in results related to the growth of algebras, as the growth of algebras is an analogue of word growth in combinatorics. Connections between the growth of groups and related rings have also been investigated. 
 A celebrated result of  Gromov theorem says that if G is a finitely generated group
and the group algebra $K[G]$ has finite Gelfand-Kirillov
dimension, then G is nilpotent by finite (see \cite{kl}, pp 139 ).  The growth of groups has been extensively studied, but there
are some basic open questions remaining in this area. One of the most tantalizing is  Grighorhuk's question of whether there exist groups with growth between polynomial and $exp {\sqrt {n}}$. Analogous questions can also be asked about algebras, although here the situation is completely different, and the question of which functions are growth functions of  associative algebras also remains open \cite{sb}. Related questions have been studied by Bell and Young \cite{by}, where they constructed nil algebras with subexponential growth over uncountable fields  (notice that nil rings with polynomial growth were previously constructed in \cite{ls}). However,  they did not give the lower bounds of the growth functions of the constructed algebras, and it was still possible that their algebras had polynomial growth. In Corollary \ref{C35}  we use Theorem \ref{A} to construct the  first example of a nil algebra with growth strictly between polynomial and exponential over an  arbitrary field, answering a question of Zelmanov \cite{zel}. 

  In a connection with dynamical systems, Newman, Shneider and Shalev introduced a notion of an entropy of an associative algebra. They defined the entropy of a graded associative algebra $A$ as $H(A)=\limsup _{n\rightarrow \infty } a_{n} ^{1/n}$ where $a_{n}=\dim A(n)$ is the vector space dimension of the $n-th$ homogeneous component \cite{nss}. Notice that we  can use Theorem \ref{A} to construct algebras with arbitrarily small but nonzero entropy satisfying prescribed relations (if the number of relations is not too big).
 
 Inspiring open questions on the growth of algebras satisfying prescribed relations were asked by Zelmanov in \cite{z}. He also suggested composing results from \cite{ls} (where examples of nil algebras with polynomial growth were found) and the  Golod-Shafarevich theorem to attack these problems, which turned out to be an inspiring idea.
 These ideas were later investigated by a student of Zelmanov, Alexander  Young, who  obtained several interesting  results on the growth of factor algebras $A/I$ where $I$ is a regimented regimented ideal, that is a special type of  ideal containing repeated paterns, in some sense resembling nil ideals. For example, 
a regimented ideal generated by a single element $f_{1}\in A $ has a form $\bigcup_{1\leq i\leq \deg f}\sum_{j=0, 2,\ldots } A((\deg f)j)fA$. In \cite{sb} it was shown that similar results can be obtained without assumption that the ideal is regimented, answering a question by Zelmanov.  
 In  more detail,  it was shown in \cite{sb} that Golod-Shafarevich algebras with a polynomial number
     of defining relations of sparse degrees can be mapped onto  algebras with polynomial growth, under mild assumption on the number of generating relations of each degree.
The main objective of this paper is to show that a more general result holds, and to bound the growth function from below.
 \begin{theorem}[Theorem A]\label{A}
  Let $K$ be an algebraically closed field, and let $A$ be the free noncommutative algebra generated in degree one by elements $x, y$. Let
  $I$ denote the ideal generated in $A$ by homogeneous
  elements $f_{1}, f_{2}, \ldots \in A$. Let $r_{n}$ denote the number of
 elements among $f_{1}, f_{2}, \ldots $ with degrees larger than
 $2^{n}$ and not exceeding $2^{n+1}$. Suppose that $r_{i}=0$ for $i<8$.  Denote $Y=\{n:r_{n}\neq 0\}$.
 \begin{itemize}
\item Suppose that there are no elements among $f_{1}, f_{2}, \ldots $ with degree $k$ if 
$k\in [2^{n}-2^{n-3}, 2^{n}+2^{n-2}]$ for some $n$.
 \item Suppose that for all $m,n\in Y\cup \{0\}$ with $m<n$ we have \\
$2^{3n+4}r_{m}^{33}<r_{n}< 2^{2^{n-m-3}}$
 and $r_{n}<2^{2^{n/2-4}}.$
\end{itemize}
  Then there is an infinitely dimensional graded  algebra $R$, such that the following holds:
\begin{itemize}
\item[1.] Algebra $A/I$ can be homomorphically mapped onto algebra $R$.
 %\[\dim R_{n}\leq 8n^{4}\prod_{i\in Y, i\leq 2log(n)}r_{i}^{32},\]
\item[2.] If $n$ is a natural number, and $k$ is maximal such that $k\leq 2 log (n)$ and $k\in Y$ then  
   \[\dim R_{n}\leq 8n^{4}r_{k}^{33},\]
where $R_{n}\subseteq R$ is the  $K$-linear space consisting of elements of degrees not exceeding  $k$.
\item[3.] Moreover, for all $j\in Y$, $j\leq log(n)$ we have  
\[{1\over 2} r_{j}^{4}\leq dim R_{n}.\]
\end{itemize} 
\end{theorem}
 We note that the assumption on the number of relations is necessary, in fact there are Golod-Shafarevich algebras without non trivial homomorphic images of subexponential or polynomial growth \cite{smok1}. An analogous result holds for Golod-Shafarevich groups, though the proof is completely different \cite{ers}.   

Observe that Theorem \ref{A}  allows us to construct algebras with various growth functions, for example with a subexponential growth, which satisfy given relations. This answers a question of Vesselin Drensky about when there exist  algebras of subexponential (and larger than polynomial) growth  satisfying prescribed relations \cite{dren}. 
 Observe that if the algebra $R$ in Theorem \ref{A} is not Jacobson radical, then we can construct a primitive algebra satisfying the prescribed relations using the fact that the Jacobson radical is the intersection of all  primitive ideals in the given ring.
 Notice also that if the number of relations $f_{1},  f_{2}, \ldots $ is finite, under the assumptions of Theorem \ref{A} the algebra $A/I$ can be mapped onto prime Noetherian algebra with linear growth (by using the main result of $\cite{s}$ and Lemma \ref{L33}). 
 Notice that Theorem \ref{A} allows us to bound the entropy of the constructed algebra $R$, with entropy defined as in \cite{nss}.  Observe that some  of the elements $f_{1}, f_{2}, \ldots $ may be zero, which allows us to apply Theorem \ref{A} to constuct algebras with various growth functions of various entropies (for some related results see \cite{kkm}).

 Recently many  open questions have arisen about algebras satisfying a prescribed number of relations  in the area of resolutions of noncommutative singularities. 
 The following question related to equivalences  of the derived category of 3-folds  in algebraic geometry was posed in \cite{dw}, Remark $5.3$:  Suppose that $F$ is the formal free algebra in two variables, and consider two relations $f_{1}, f_{2}$ such that if we write both $f_{1}$ and $f_{2}$ as a sum of words, each word has degree two or higher. Denote $I$ to be the two sided ideal generated by $f_{1}$ and $f_{2}$. Is it true that  if $F/I$ is finite dimensional,  it cannot be commutative? 
 In chapter $4$  we answer this question and prove a slightly more general result, namely:
\begin{theorem}\label{2} Let $K$ be a field.
 Let $F$ be either the free associative $K$- algebra on the set of free generators $X=\{x_{1}, x_{2}, \ldots ,x_{n}\}$  over the field $K$  or $F$ be the formal free power series algebra over $K$  in $n$ variables
 $x_{1}, \ldots , x_
{n}$.  Let $d=2$ for $n=2$, and let $d={n(n+1)\over 2}-2$ for $n>2$. Consider 
$d$ relations $f_{1}, f_{2}, \ldots , f_{d}\in F$ such that if we write each of the $f_{1}, \ldots , f_{d}$ as a sum of 
 words then each word has degree two or higher. If  $F/I$ is finite dimensional then it cannot be commutative. 
\end{theorem}
 Notice that relations $f_{1}, f_{2}, \ldots $ in Theorem \ref{2} need not be homogeneous. 
 Wemyss constructed several not commutative finitely dimensional algebras defined by $f_{1}$ and $f_{2}$ as above. 
    These examples also show that under assumptions of Theorem \ref{2}  it is possible to obtain finitely dimensional algebra $F/I$ which is not commutative.
  Let $F$ be a free algebra in $n$ generators $x_{1}, \ldots , x_{n}$, and let $I$ be the ideal generated by ${n(n+1)\over 2}$ relations 
  $x_{i}x_{j}-x_{j}x_{i}$ and $x_{i}^{2}$ for $i\neq j$, $i,j\leq n$, $n>2$. Then $F/I$ is finitely dimensional and commutative.  On the other hand Theorem \ref{2} assures that the algebra $F/I$ with $I$ generated by less than ${n(n+1)\over 2}-1$ 
 relations cannot be commutative and finitely dimensional. We don't know if there is finitely dimensional, commutative algebra $F/I$ with $I$ generated by ${n(n+1)\over 2}-1$ relations.
 For some recent results related to power series rings and polynomial rings with a small number of defining relations, see \cite{mz1, mz2, mz3, mz4, n, ziem}.
 
\subsection{Notation and definition of the Golod-Shafarevich algebras}\label{S01}

 We use the same notation as in \cite{sb}, \cite{s}.
In what follows, $K$ is a countable, algebraically closed field and
$A$ is the free $K$-algebra in two non-commuting indeterminates $x$
and $y$. By a graded algebra we mean an algebra graded by the additive semigroup of natural numbers.
The set of monomials in $\{x,y\}$ is denoted by $M$ and, for
each $k\geq 0$, its subset of monomials of degree $k$ is denoted by
$M(k)$.  Thus, $M(0)=\{1\}$ and for $k\geq1$ the elements in $M(k)$
are of the form $x_1\cdots x_k$ with all $x_i\in \{x, y\}$. The span
of $M(k)$ in $A$ is denoted by $A(k)$; its elements are called
\emph{homogenous polynomials of degree $k$}. More generally, for any
 subsetset $X$ of a graded algebra, we denote by $X(k)$ its subset of homogeneous
elements of degree $k$. The \emph{degree} $\deg f$ of an element $f \in A$ is the least
$k\ge0$ such that $f \in A(0) + \cdots + A(k)$. Any $f\in A$ can be
uniquely written in the form $f=f_0+f_1+\cdots+f_k$ with each $f_i\in
A(i)$. The elements $f_i$ are the \emph{homogeneous components} of
$f$.  A (right, left, two-sided) ideal of $A$ is \emph{homogeneous} if
it is spanned by its elements' homogeneous components.  If $V$ is a
linear space over $K$, we denote by $\dim V$ the dimension of $V$
over $K$ (for information about the growth of algebras, see \cite{kl}).
 All logarithms are in base $2$.

We now recall the definition of Golod-Shafarevich algebras.
 Let $A$ be a free associative  algebra in $d$ generators $X=\{x_1, x_2, \ldots ,x_d\}$, with generators $x_i$ of degree
$1$. Suppose we are given homogeneous elements $f_{1}, f_{2}, \ldots \in A$, with $r_{i}$ elements of  degree $i$ among $f_{1},f_{2}, \ldots $ (and $r_{1}=0$). 
 Let $I$ be the ideal of $A$ generated by $f_{1}, f_{2}, \ldots $. Denote $R=A/I$. Recall that  $H_{R}(t)=\sum_{i=0}^{\infty }\dim R(i)t^{i}$ is called the Hilbert series of algebra $R$.
 Golod and Shafarevich proved in \cite{gs}
that $H_{R}(t)(1-mt+\sum_{i=2}^{\infty} r_{i}t^{i})\geq 1$ holds
co\"efficient-wise.  An
algebra admitting such a presentation is called a Golod-Shafarevich algebra. 
 In a survey by Zelmanov \cite{zel} it was noted that if there exists a number $t_0\in(0,1)$
such that $1-mt+\sum_{i=2}^{\infty} r_{i}t^{i}$ converges at $t_0$ and
$1-mt_{0}+\sum_{i=2}^{\infty} r_{i}t_{0}^{i}<0$, then $A/I$ is infinite dimensional.

 {\em  Notice that  if 
  for  infinitely many natural numbers $m$, there is a number $t_{m}>0$ such that
 $1-dt_{m}+\sum_{i=2}^{m } r_{i}{t_{m}}^{i}<0$ then algebra $A/I$ is infinitely dimensional.} 
Although this is an easily found consequence of  Zelmanov's observation,  we did not find it in the literature, so we provide a short proof below for the convenience of the reader:
\begin{proof}
 Fix number $m$ and let $I_{m}$ be the ideal of $A$  generated by all elements $f_{i}$ which have degrees not exceeding $m$. If there is $t_{m}>0$ such that $1-dt_{m}+\sum_{i=2}^{m } r_{i}{t_{m}}^{i}<0$ then by the Golod-Shafarevich theorem and the previous remarks from \cite{zel} $A/I_{m}$ is infinitely dimensional. 

Since $I$ is a graded ideal and $I=\bigcup_{i=2}^{\infty }I_{m}$ it follows that $A/I$ is infinite dimensional. Indeed a graded and finite dimensional algebra is nilpotent. If $A/I$ is finitely dimensional them $A(n)\in I$ for some $n$, and since $I$ is graded  $A(n)\subseteq I_{n}$, so $A/I_{n}$ is finitely dimensional, a contradiction.
\end{proof}
\section{General construction}\label{S1}
 The general construction is similar to that in \cite{ls}, \cite{sb}, \cite{s}. 
  We will use the notation from \cite{s}.  In this chapter we will review some basic concepts from chapter 1 in \cite{s}.

 Let $K$ be a field and
 $A$ be a free $K$-algebra generated in degree one by two elements $x$, $y$.
 Suppose that subspaces $U(2^m),V(2^m)$ of $A(2^m)$ satisfy,
  for every $m\geq 1$, the following properties:
  \begin{itemize}
  \item[1.] $V(2^m)$ is spanned by monomials;\label{prop:3}
  \item[2.] $V(2^m)+U(2^m)=A(2^m)$ and $V(2^m)\cap U(2^m)=0$;
  \item[3.] $A(2^{m-1})U(2^{m-1})+U(2^{m-1})A(2^{m-1})\subseteq U(2^m)$;
  \item[4.] $V(2^m)\subseteq V(2^{m-1})V(2^{m-1})$, where for $m=0$ we set $V(2^{0})=Kx+Ky$, $U(2^{0})=0$.
  \end{itemize}
 Following \cite{ls}, define a graded subspace $ E$ of $A$
by constructing its homogeneous components $ E(k)$ as
follows. Given $k\in N$, let $n\in N$ be such that $2^{n-1}\leq
k<2^n$. Then $r\in E(k)$ precisely if, for all
$j\in\{0,\dots,2^{n+1}-k\}$, we have $A(j)rA(2^{n+1}-j-k)\subseteq
U(2^n)A(2^n)+A(2^n)U(2^n)$. More concisely,
\begin{equation}
   E(k)=\{r\in A(k)\mid ArA\cap A(2^{n+1})\subseteq U(2^n)A(2^n)+A(2^n)U(2^n)\}.
\end{equation}
Set then $ E=\bigoplus_{k\in N} E(k)$.

\begin{lemma}[Lemma 1.1, \cite{s}]\label{L1}
  The set $ E$ is an ideal in $A$. Moreover, if all sets $V(2^{n})$ are nonzero, then
 algebra $A/ E$ is infinite dimensional over $K$.
\end{lemma}
 Notice that the proof of the first claim is the same as in Theorem 5 \cite{ls}.

In the next chapters we will construct appropriate sets $U(2^{n})$, $V(2^{n})$, so that the algebra $R=A/E$ will satisfy thesis of Theorem \ref{A}.

Let $k\in N$ be given. We write it as a sum of increasing powers of $2$,
namely $k=\sum_{i=1}^t 2^{p_i}$ with $0\leq p_1 < p_2 < \ldots <
p_t$. Following \cite{sb}, \cite{s} we set
\begin{align}
  U^<(k) &= \sum_{i=0}^t A(2^{p_1}+\cdots+2^{p_{i-1}})U(2^{p_i})A(2^{p_{i+1}}+\cdots+2^{p_t}),\label{def:U<}\\
  V^<(k) &= V(2^{p_1})\cdots V(2^{p_t}),\\
  U^>(k) &= \sum_{i=0}^t A(2^{p_t}+\cdots+2^{p_{i+1}})U(2^{p_i})A(2^{p_{i-1}}+\cdots+2^{p_1}),\\
  V^>(k) &= V(2^{p_t})\cdots V(2^{p_1}).\label{def:V>}
\end{align}
 The following lemma corresponds to Lemma $1.4$ in \cite{s}.
\begin{lemma}[\cite{ls}, pp. 993--994]\label{L2}
  For all $k\in N$ we have $A(k)=U^<(k)\oplus V^<(k)=U^>(k)\oplus
  V^>(k)$.

  For all $k,\ell\in N$ we have $A(k)U^<(\ell)\subseteq U^<(k+\ell)$ and $U^>(k)A(\ell)\subseteq U^>(k+\ell)$.
\end{lemma}
 The following proposition corresponds to Proposition $1.1$ in \cite{s}.
\begin{proposition}[Theorem 11, \cite{ls}]\label{P1}
  For every $k\in N$ we have
  \[\dim A(k)/ E(k)\leq \sum_{j=0}^k\dim V^<(k-j)\dim V^>(j),\]
 where we set $\dim V^{>}(0)=\dim V^{<}(0)=1$.
\end{proposition}
\begin{proof} The proof is the same as the proof of Theorem $11$ in \cite{ls}, or the proof of Theorem $5.2$ in \cite{lsy}.
\end{proof}
\section{Chapters 3, 4, 5 and  6  from \cite{s}}\label{S2}

In this section we assume that $r_{i}$ and $Y, f_{1}, f_{2}, \ldots $ are as in Theorem $A$. We moreover assume that there are 
 natural numbers $\{e_{i}\}_{i\in Y}$ which satisfy the following conditions: for all $n\in Y$: $1\leq e(n)\leq n-1$, sets $S_{n}=\{n-1-e(n), n-1\}$ are disjoint and $r_{n}\leq 2^{t(n)}$ where $t_{n}=2^{e(n)-1}-3n-4-\sum_{k\in Y, k<n}2^{e(k)+2}$.
 Let $S=\bigcup _{n\in Y}\{n-1-e(n), n-1\}$.

We will use some results from \cite{s} and assume that the reader is  familiar with chapters  3, 4 and 5  of it.
 Also our Section \ref{S1} corresponds with section $1$ in \cite{s}, but we omitted Theorems $1.2$ and $1.3$ from \cite{s}, as they are not relevant to this paper.
 Because we are using the same notation, we can use  chapters $3,4,5$ in \cite{s}  without modification (similar results with different notation were also studied in \cite{sb}).

We have slightly weaker assumptions on elements $f_{i}$ in Theorem $A$   than those  in Theorem $0.1$ in \cite{s}, namely that we can have an infinite number of nonzero $f_{i}$  in Theorem $A$ and  only assume that there are no $f_{i}$ of degree $k$ if $k\in [2^{n}-2^{n-3}, 2^{n}+2^{n-2}]$  for some $n$.
 In \cite{s} it was assumed that the number of $f_{i}$ is finite and that  there are no $f_{i}$ of degree $k$ if $k\in [2^{n}-2^{n-3}, 2^{n}+2^{n-1}]$  for some $n$.
 However,  this does not influence the proofs in chapters 3, 4 and 5 in \cite{s}.
On the other hand, our assumption in Theorem $A$  that  for all $n,m\in Y\cup \{0\}$ with $m<n$ we have $2^{3n+4}r_{m}^{33}< r_{n}< 2^{2^{n-m-3}}$ is a stronger  assumption than $2^{3n+4}\prod_{i<n, i\in Y}r_{i}^{32}< r_{n}< 2^{2^{n-m-3}}$ in Theorem $0.1$ \cite{s} (see Lemma \ref{L33} for an explanation), hence we can use the results from \cite{s}.

 We slightly modify results from section $6$ in \cite{s} as follows, to suit our new assumptions.
  We begin with the following
lemma, which also generalizes Lemma $4.1$ from \cite{sb}.
\begin{lemma}[Modified Lemma $6.1$, \cite{s}] \label{L21} Let $n$ be a natural number.
Suppose that, for all $m<n$ with $m\in Y$, we constructed sets $F(2^{m})$, and for every $m<n$ we are given subspaces $V(2^{m})$, $U(2^{m})$ of  $A(2^{m})$ which satisfy
the following properties  (with
$Y$, $\{e(i)\}_{i\in Y}$ defined as at the beginning of this section):
\begin{itemize}
\item[1.] $\dim V(2^{m})=2$ if $m\notin S$;
\item[2.] $\dim V(2^{m-e(m)-1+j})=2^{2^{j}}$ for all $m\in Y$ and all $0\leq j\leq e(m)$;
\item[3.] $V(2^{m})$ is spanned by monomials;
\item[4.] $F(2^{m})\subseteq U(2^{m})$ for every $m\in Y$;
\item[5.] $V(2^{m})\oplus U(2^{m})=A(2^{m})$;
\item[6.] $A(2^{m-1})U(2^{m-1})+U(2^{m-1})A(2^{m-1})\subseteq U(2^{m})$:
\item[7.] $V(2^{m})\subseteq V(2^{m-1})V(2^{m-1})$.
\end{itemize}
  Consider all $f\in A(k)\cap \{f_{1}, f_{2}, \ldots \}$ with
  $2^n+2^{n-2}\leq k\leq 2^n+2^{n-1}+2^{n-2}$.
  Then there exists a linear $K$-space $F'(2^n)\subseteq A(2^n)$ with the
  following properties:
  \begin{itemize}
  \item $0<\dim F'(2^n)\leq \frac12\dim V(2^{n-1})^2$;
  \item for all $i,j\ge0$ with $i+j=k-2^n$ and for every $f\in A(k)\cap \{f_{1}, f_{2}, \ldots \}$, we have $f\in
    A(i)F'(2^n)A(j)+U^<(i)A(k-i)+A(k-j)U^>(j)$ with the sets
    $U^<(i),U^>(i)$ defined in Section \ref{S1}.
  \end{itemize}
 \end{lemma}
 \begin{proof}
  The proof is the same as the proof of Lemma $6.1$, \cite{s}. 
\end{proof}
 Observe that properties $1-7$ in Lemma \ref{L21} correspond to properties $1-7$ in Theorem $4.1$ \cite{s}.

  We note that there is a small typo in Theorem $4.1$ (2) \cite{s}, and it should state `` $m\in Y$'' instead of ``$m\in Z$.''
 Theorem  $4.1$ (1)-(7) \cite{s}  was invented and used in \cite{ls} but with a different notation.  
\begin{lemma}[Modified Lemma $6.2$, \cite{s}]\label{L22}
  Suppose that sets $U(2^m), V(2^m)$  were already constructed for all
  $m<n$, and sets $F(2^{m})$ were already constructed for all $m<n$ with $m\in Y$, and satisfy the conditions of Theorem $A$. Let $n\in Y$. Let
  $F=\{f_{1}, f_{2}, \ldots \}$ be as in Theorem $A$.
    Define a $K$-linear
  subspace $Q\subseteq A(2^{n+1})$ as follows:
  \[Q=\sum_{f\in F:2^{n}+2^{n-2}\leq \deg f\leq 2^{n}+2^{n-1}+2^{n-2}}\quad
  \sum_{i+j=2^{n+1}-\deg f}V^{>}(i)fV^{<}(j).\]
  Then $\dim Q\leq \frac14(\frac12\dim V(2^{n-1})^2-2)$.
\end{lemma}
\begin{proof} (modified to suit our paper)
 By Lemma $5.3$ in \cite{s}, the inner sum has dimension at most
 \[ 2^{2n}(\prod_{k<n, k\in Y}2^{2^{e(k)+1}})^{2}[V(2^{n-1})]^{2}/2^{2^{e(n)-1}}.\]
  Summing over all $i+j=2^{n+1}-\deg f$
  multiplies by a factor of at most $2^{n}$ (because $2^{n+1}-\deg f\leq 2^{n-1}$); summing over all $f\in
  \{f_{1}, f_{2}, \ldots \}$ with degrees between $2^{n}+2^{n-2}$ and
   $2^{n}+2^{n-1}+2^{n-2}$ multiplies by $r_{n}$.
   Therefore, \[\dim Q \leq r_{n}2^{3n}(\prod_{k<n, k\in
Y}2^{2^{e(k)+1}})^{2}[V(2^{n-1})]^{2}/2^{2^{e(n)-1}}.\]
    By assumption that $r_{n}\leq 2^{t_{n}}$ from the beginning of this section, we get
    $\dim Q  \leq \frac1{16}\dim V(2^{n-1})^2$.

  Observe now that $\frac14\dim V(2^{n-1})^2\leq \frac12\dim V(2^{n-1})^2-2$,
  because $V(2^{n-1})^2=(2^{2^{e(n)}})^2\geq 2^{2^2}\geq 16$.  We get
  $\dim Q \leq\frac14(\frac12\dim V(2^{n-1})^2-2)$ as required.
\end{proof}
We will also use Proposition $6.1$ \cite{s} and repeat it for the convenience of the reader. This proposition previously appeared with different notation as  Proposition 3.5 in  \cite{sb}.

\begin{proposition}[Proposition $6.1$, \cite{s}]\label{P21}
Let $K$ be an algebraically closed field.
  With notation as in Lemma \ref{L21} there is a linear $K$-space
  $F(2^n)\subseteq A(2^n)$ satisfying $\dim
  F(2^n)\le \dim V(2^{n-1})^2-2$ and \[F'(2^{n})\subseteq
  F(2^{n})+U(2^{n-1})A(2^{n-1})+A(2^{n-1})U(2^{n-1}).\] Moreover, for all $f\in \{f_{1}, f_{2}, \ldots \}$ with
  $\deg f\in\{2^n+2^{n-2},\dots,2^n+2^{n-1}+2^{n-2}\}$ we have
  \begin{multline*}
    AfA\cap A(2^{n+1})\subseteq A(2^n)F(2^n)+F(2^n)A(2^n)\\
    +A(2^{n-1})U(2^{n-1})A(2^n)+A(2^n)U(2^{n-1})A(2^{n-1})\\
    +U(2^{n-1})A(2^n+2^{n-1})+A(2^n+2^{n-1})U(2^{n-1}).\\
  \end{multline*}
\end{proposition}

\begin{proof}(modified to suit our paper) Proof of this proposition under the assumptions of our section is the same, but we use modified  Lemma $6.2$ (Lemma \ref{L21}) as above, instead of using Lemma $6.2$ from \cite{s}. Similarly we use modified Lemma $6.2$ (Lemma \ref{L22}) instead of Lemma $6.2$. 
\end{proof}

We will also use Proposition $5.1$ from \cite{s}.

\begin{proposition}[Proposition 5.1, \cite{s}]\label{P22}
  Let $\alpha = 2^{p_1} + \cdots + 2^{p_{t}}$ be a natural number in the binary form.
   Then $[V^>(\alpha)]<2\alpha \prod_{i\leq m, i\in Y }2^{2^{e(i)+1}}$, where $m$ is maximal
  such that $\sum_{p_i\in \{m-e(m)-1,\dots, m-1\}}2^{p_{i}}$ is nonzero.
  \end{proposition}

 \section{Dimensions of images of Golod-Shafarevich algebras}\label{S3}
  In this section we will prove Theorem $A$. We also obtain the following result, which is more general than Theorem $A$.
 \begin{theorem}\label{T31} Assume that the assumptions of Theorem $A$ hold and that we use the same notation as in Theorem $A$. Assume that, for each $n\in Y$, there is a
   natural number $1\leq e(n)< {n\over 2}-1$ such that, for all
   $n\in Y$, sets $S_{n}=\{n-1-e(n), n-1\}$ are disjoint and
   $$r_{n}2^{3n+4}\prod_{k<n, k\in
   Y}2^{2^{e(k)+2}}\leq 2^{2^{e(n)-1}}$$
 (where we put $\prod_{ k<n, k\in Y }2^{2^{e(k)+2}}=1$ if there are no $k<n$ with $k\in Y$). 
    Then $A/I$ can be mapped onto an infinite dimensional, graded algebra $R$ which satisfies
  $\dim R(n)\leq 8n^{3}\prod_{i\in Y, i\leq 2log(n) }2^{2^{e(i)+2}}.$
 \end{theorem}
\begin{proof} Let $e(i)$, $Y$ be as in the assumptions of this theorem.
 We will construct sets $U(2^{n})$, $V(2^{n})$, $F(2^{n})$ satisfying properties $1-7$ from Theorem $4.1$ \cite{s}
 (or equivalently properties $1-7$ from Lemma \ref{L21})  applied for $e(n)$ as above. We start the induction with $U(2^0)=F(2^0)=0$ and
  $V(2^0)=Kx+Ky$. Then, assuming that we constructed
  $U(2^m), V(2^m)$ for all $m<n$, if $n\in Y$ we construct $F(2^n)$ using
  Proposition \ref{P21}, and if $n\notin Y$ we set $F(2^{n})=0$. We then construct $U(2^n)$, $V(2^n)$
  using Theorem $4.1$ (1)-(7) in \cite{s}. Let $E$ be defined as in Section \ref{S1}.
  By Lemma \ref{L1}, the set E is an ideal in A and  $A/E$ is an infinite dimensional algebra. Because $E$ is homogeneous, $A/E$ is graded.
 We will now show
  that $R=A/E$ is a homomorphic image of $A/I$. We need to show that $I\subseteq E$, that is that elements $f_{1}, f_{2}, \ldots \in E$. Let $f\in A(k)$ be one of these elements. By Lemma \ref{L21}  and Proposition \ref{P21}, and because  $F'(2^{n})\subseteq F(2^{n})+U(2^{n-1})A(2^{n-1})+A(2^{n-1})U(2^{n-1})\subseteq U(2^{n})$, we get that $f$ satisfies  the assumptions of Theorem $3.1$ \cite{s}. Therefore, and by thesis of Theorem $3.1$ \cite{s}, we have $f\in E$, as required.

  Recall that we assumed that $e(n)\leq {n\over 2}-1$. We will apply Proposition \ref{P22} for $\alpha =n$. Let $n=\sum_{i=1}^{t}2^{p_{i}}$, and let $m$ be as in Proposition \ref{P22}. Recall that $m$ is maximal such that $\sum_{p_{i}\in\{m-e(m)-1,....,m-1\}}2^{p_{i}}\neq 0$.
  We will show that $m\leq 2log(n)$. 
  Note that $log (n)\geq p_{t}$ and so  $m\leq 2log (n)$,
  because otherwise $m>2p_{t}$ would imply $m-e(m)-1>p_{t}$, as $e(m)\leq {m\over 2}-1$ by assumption,
  so $\sum_{p_{i}\in\{m-e(m)-1,....,m-1\}}2^{p_{i}}= 0$, a contradiction. Therefore, $m\leq 2 log (n)$.

  Clearly, $\dim R(n)=\dim A(n)/E(n)$.
  By Lemma \ref{L1} we have
  \[\dim A(n)/ E(n)\leq \sum_{j=0}^{n}\dim V^<(n-j)\dim V^>(j).\]
 By Proposition \ref{P22}, we get
 $\dim A(n)/E(n)\leq (n+1)(2n \prod_{i\in Y, i\leq 2log(n) }2^{2^{e(i)+1}})^{2}\leq 8n^{3}\prod_{i\in Y, i\leq 2log(n) }2^{2^{e(i)+2}}$, as required.
\end{proof}
\begin{lemma}\label{L32}
 With the assumptions as in Theorem \ref{T31} we in addition have 
  $\dim R_{n}>{1\over 2}2^{2^{e_{l}}}$, for all $l\in Y$, $l\leq log(n)$, where $R_{n}$ is the linear space of elements with degrees not exceeding $n$ in $R$. 
\end{lemma}
\begin{proof} Let $n$ be a natural number and let $l\in Y$, $l\leq log(n)$. By the construction from Theorem $4.1$ \cite{s} we get that 
 $V(2^{l-1})$ has $2^{2^{e(l)}}$ elements (this can be also seen using property  (2) in Lemma \ref{L21}). Indeed by Theorem $4.1$ (2) \cite{s} we get $\dim V(2^{m-e(m)-1+j})=2^{2^{j}}$ for all $m\in Y$ and all $0\leq j\leq e(m)$ (we note that there is a small typo in the statement of Theorem $4.1$ (2) \cite{s}, saying $m\in Z$ instead of $m\in Y$). Therefore,  and since $l\in Y$ we get $\dim V(2^{l-1})=\dim V(2^{l-e(l)-1+e(l)})=2^{2^{e(l)}}$. 
  By assumption $e(l)\geq 1$ hence $\dim V(2^{l})\geq 4$. By Theorem $4.1$ (1) and (2)  \cite{s}, we get 
$V(2^{l-2})V(2^{l-2})=V(2^{l-1})$ and $\dim V(2^{l-2})=\dim V(2^{l-e(l)-1+e(l)-1})=2^{2^{e(l)-1}}$ (it is also true if $e(l)=1$ by  Theorem $4.1$ (1) \cite{s}). Set $V(2^{l-1})$ is generated by monomials, so at least half of these monomials end with $x$ or at least half of these monomials ends with $y$. We can assume without restricting the generality that the former holds.
 Let $v_{1}, \ldots , v_{p}\in A(2^{l-1}-1)$  be  monomials such that $v_{1}x, \ldots v_{p}x\in V(2^{l-1})$ and $p\geq  {1\over 2}\dim V(2^{l-1})$. We claim that images of  elements $v_{1}, \ldots , v_{p}$ in $A/E$ are linearly independent over $K$.  As $A/E=R$, and $p\geq {1\over 2}\dim V(2^{l-1})={1\over 2}2^{2^{e(l)}}$, this statement implies the thesis of our lemma.  It remains to show that  
 if  $v=\sum_{i=1}^{p} \alpha_{i}v_{i}$ for some $\alpha _{i}\in K$ then $v\notin E$ (unless all $\alpha _{i}$ are zero). Suppose on the contrary that $v\in E$. Let $w\in V(2^{l-2})$.  By the definition of $E$, and since $v\in E$, we get $wvxw\in U(2^{l-1})A(2^{l-1})+A(2^{l-1})U(2^{l-1})$. 
  Write $v_{i}x=u_{i}z_{i}$ for $u_{i},z_{i}\in V(2^{l-2})$.  Now $wu_{i}, z_{i}w\in V(2^{l-1})$ (because $V(2^{l-2})V(2^{l-2})=V(2^{l-1})$), and so  $wvxw\in V(2^{l-1})V(2^{l-1})$. 
By the property $2$ from the beginning of chapter $1$ we get that $V(2^{l-1})V(2^{l-1})\cap (U(2^{l-1})A(2^{l-1})+A(2^{l-1})U(2^{l-1}))=0$. It follows that $wvxw=0$ in $A$, and so $vx=0$, hence all $\alpha _{i}=0$,
 which proves our claim.
\end{proof}
\begin{lemma}\label{L33}
 Suppose that elements $r_{i}$ are as in Theorem $A$. Then $$2^{3n+4}\prod_{i<n, i\in Y}r_{i}^{32}< r_{n}< 2^{2^{n-m-3}}.$$
Moreover for all $m, n\in Y\cup \{0\}$, with $m<n$ we have $r_{n}< 2^{2^{n-m-3}}$. 
\end{lemma}  
\begin{proof} Observe first that by the assumptions of Theorem $A$  we have 
$2^{3n+4}r_{m}^{33}<r_{n}< 2^{2^{n-m-3}}$. By a similar argument we get that  $r_{m}>r_{m'}^{33}$ for $m'<m$, with $m, m'\in Y$.
Consequently, for any $m_{1}>m_{2}>m_{3}\ldots $ with $m_{1}, m_{2}, \ldots \in Y$ the following holds:   
 $r_{m_{1}}^{33}=(r_{m_{1}}^{32})r_{m_{1}}> r_{m_{1}}^{32}r_{m_{2}}^{33}>r_{m_{1}}^{32}r_{m_{2}}^{32}r_{m_{2}}>r_{m_{1}}^{32}r_{m_{2}}^{32}r_{m_{3}}^{33}\geq \ldots  $.
 It follows that $$2^{3n+4}\prod_{i<n, i\in Y}r_{i}^{32}< r_{n}< 2^{2^{n-m-3}}.$$
\end{proof}

\begin{lemma}\label{L34} Let $Y$ be a subset of the set of natural numbers and let $\{r_{i}\}_{i\in Y}$ be a sequence of natural numbers which 
 satisfy the assumptions of Theorem $A$.
  Then there are natural numbers $\{e(n)\}_{n\in Y}$ such that for all $n\in Y$:  
 $1\leq e(n)\leq {n\over 2}-1$ and  sets
 $S_{n}=\{n-1-e(n), n-1\}$ are disjoint and $r_{n}2^{3n+4}\prod_{k<n, k\in
  Y}2^{2^{e(k)+2}}\leq 2^{2^{e(n)-1}}.$ 
Moreover  $2^{2^{e_{n}}}\geq  r_{n}^{4}$. 
\end{lemma}
 \begin{proof} The proof of the first statement is the same as proof of Lemma $7.2$ in \cite{s} (because 
   by Lemma \ref{L33} elements $\{r_{i}\}_{i\in Y}$  satisfy assumptions of Theorem $0.1$ \cite{s}).
 Recall, that in particular for  each $i$,  $e(i)$ is such that 
 $2^{2^{e(i)-3}}\leq r_{i}<2^{2^{e(i)-2}}$.
  Observe that then $e(n)<{n\over 2}-1$ because 
   by assumptions of Theorem $A$,  $r_{n}<2^{2^{n/2-4}}$, and hence  
 $2^{2^{e(n)-3}}\leq r_{n}<2^{2^{n/2-4}}$ which implies $e(n)\leq {n\over 2}-1$.

 To prove the second statement observe that $2^{2^{e_{n}}}=2^{2^{(e_{n}-2)+2}}=2^{2^{e_{n}-2}4}=(2^{2^{e_{n}-2}})^{4}$.
  Recall that 
  $2^{2^{e(i)-2}}\geq r_{i}$. Consequently  $2^{2^{e_{n}}}\geq r_{n}^{4}$, as required.
\end{proof}

{\bf Proof of Theorem \ref{A}} 
By Lemma \ref{L34}, we can find $e(i)$ satisfying the assumptions of Theorem \ref{T31}, 
   and by the thesis of Theorem $4.1$ \cite{s} we get that $A/I$ can be mapped onto an infinite dimensional, graded algebra $R$ which satisfies
  $\dim R(n)\leq 8n^{3}\prod_{i\in Y, i\leq 2log(n) }2^{2^{e(i)+2}}$.
   Recall that in Lemma \ref{L34} we assume that
  $2^{2^{e(i)-3}}\leq r_{i}<2^{2^{e(i)-2}}$, for each $i\in Y$.
 Consequently,  $2^{2^{e(i)+2}}=
 2^{2^{(e(i)-3)+5}}=2^{2^{e(i)-2}32}=(2^{2^{e(i)-3}})^{32}\leq
 r_{i}^{32}$. Therefore,  $\prod_{i\in Y, i\leq 2log(n)
 }2^{2^{e(i)+2}}\leq \prod_{i\in Y, i\leq 2log(n)
 }r_{i}^{32}$.

 This implies that $\dim R(n)\leq 8n^{3}\prod_{i\in Y, i\leq 2log(n)
 }r_{i}^{32}$. Recall that $R(n)$ is the homogeneous space of elements with degree $n$, and $R_{n}$ is the subspace of $R$ consisting of all elements with degrees not exceeding $n$.  
 It follows that $\dim R_{n}\leq n (sup_{k\leq n}\dim R(k))$. Consequently, $\dim R_{n}\leq  8n^{4}\prod_{i\in Y, i\leq 2log(n)
 }r_{i}^{32}$, as required.

 To prove the lower bound observe that $\dim R_{n}>{1\over 2}2^{2^{e(l)}}$ for all $l\leq log(n)$, by Lemma \ref{L32}. Consequently, 
 by the last part of Lemma \ref{L34}, we get  $\dim R_{n}>{1\over 2}2^{2^{e(l)}}\geq {1\over 2}r_{l}^{4}$ for all $l\leq log(n)$.

\begin{corollary}\label{C35}
 Let $K$ be a field. Then there is a graded nil algebra with neither polynomial nor exponential growth. Moreover, this algebra is generated  
 by two elements of degree one.
\end{corollary}
\begin{proof} Let $A$ be the free noncommutative algebra generated by two generators $x$ and $y$ of degree one. Let $M$ be the set of monomials in $A$. Let $n$ be a natural number. Fix $(a_{1}, \ldots , a_{n})$ with all $a_{i}\in M$.  Let $L(a_{1}, \ldots ,a_{n})^{q}$ be the linear space spanned by elements $s^{q}$ where $s\in Ka_{1}+Ka_{2}+\ldots Ka_{n}$.
   Observe that  $L(a_{1}, \ldots ,a_{n})^{q}$ is spanned as a linear $K$-space by its  $q^{n}$ elements (this can be seen by writing each of these elements as a sum of homogeneous components). 

 Fix $(a_{1}, \ldots , a_{n})$ with all $a_{i}\in M$. Let $t>\deg a_{i}$, $r>3^{6t}$, $w=2r$.   By Theorem $2$ \cite{smok}  there exists a set $F\subseteq A(r)$ such that 
\begin{itemize}
\item[1.] $L(a_{1}, \ldots ,a_{n})^{20r}$ is contained in the ideal generated by $F$ in $A$.
\item[2.] cardinality of $F=F(a_{1}, \ldots ,a_{n})$ is less than $(20r)^{n}(r3^{2t}t^{2})$.
\end{itemize}
    (Note that the fact that the algebra $A$ in \cite{smok} is generated by three monomials does not influence the proof.)
 Observe that the cardinality of $F$ is less than $(20r)^{n+1}3^{4t}<(20r)^{n+1}r<(20r)^{n+2}$.  If $r=2^{m}+2^{m-1}$ we get that the cardinality of $F$ is less that $40^{8nm}$.

 The set $S$ of all tuples $(a_{1}, a_{2}, \ldots ,a_{n})$ with $a_{i}\in M$, $n\geq 1$  is countable.
 Therefore, there exists an injective  function $F:S\rightarrow  N$  ($N$ is the set of natural numbers). We can assume that
\begin{itemize}
 \item $f(a_{1}, a_{2}, \ldots ,a_{n})=2^{s}+2^{s-1}$ for some $s\in N$.
 \item  If $ 2^{m}+2^{m-1}=f(a_{1}, a_{2}, \ldots ,a_{n})$ and $2^{m'}+2^{m'-1}=f(b_{1}, b_{2}, \ldots ,b_{n'})$ for some $m<m'$ then  $200m^{3}<m'$, $n<m$, $n'<m'$, $m^{3}<2^{m/2-4}$ and $m>100$.
\end{itemize}
 Observe then that $r_{m}^{33}2^{3m'+4}<r_{m'}$,
  where $r_{m}=(40)^{8m^{3}}$ and
 $r(m')=(40)^{8(m')^{3}}$. 
 Observe also that $r_{m'}<2^{2^{m'/2-4}}<2^{2^{m'-m-3}}$ since $m>100$ and $200m^{3}<(m')$.

Let $f(a_{1}, \ldots ,a_{n})=2^{m}+2^{m-1}=r$. Then 
$$L(a_{1}, \ldots ,a_{n})^{20r}= 
L(a_{1}, \ldots ,a_{n})^{20f(a_{1}, \ldots  ,a_{n})}$$  is in the ideal generated by some set $F(a_{1}, \ldots ,a_{n})\subseteq A(r)$, with  cardinality  $(40)^{8m^{3}}$. 
  Denote the cardinality of this set as $r_{m}$. Denote $Y=\{n:r_{n}\neq 0\}$, and notice that $Y$, $r_{n}$ satisfy the assumptions of Theorem \ref{A} (Theorem $A$) with relations $f_{i}$ from all the sets $F(a_{1}, \ldots ,a_{n})$.
 Then, by Theorem \ref{A} there is algebra $R$ in which all elements from sets $F(a_{1}, \ldots ,a_{n})$ are zero 
 and hence this $K$-algebra is nil (even if $K$ is uncountable). We claim that the growth of this algebra is neither polynomial nor exponential.

By Theorem \ref{A} (Theorem $A$), for infinitely many $m$, $\dim R_{2^{m+1}}>r_{m}^{4}>40^{8m^{2}}$ therefore $R$ cannot have polynomial growth.

On the other hand $\dim R_{n}<8n^{4}r_{k}^{33}$ for some $k<2log(n)$. Observe then that  $r_{k}<40^{32log(n)^{3}}<2^{200 log (n)^{3}}$. 
 Consequently  $\dim R_{n}<8n^{4}2^{200 (log (n))^{3}}$ for all $n$. If $R$ had exponential growth then for some $c>1$ we would have $\dim R_{n}>c^{n}$. This would imply $c^{n}<2^{400(log (n))^{3}}$, a contradition. It follows that $R$ has smaller growth than exponential, as required.
\end{proof}
\section{Applications in algebraic geometry}\label{S4}

 Here we prove Theorem \ref{2}
 In the proof we will use the following lemmas.

\begin{lemma}\label{L41} 
 Let $K$ be a field and let $F$ be either the free commutative $K$-algebra on the set of free generators $X=\{x_{1}, x_{2}, \ldots , x_{n}\}$ or $F$ be the formal commutative power series algebra over $K$ in $n$ variables $x_{1}, x_{2}, \ldots ,x_{n}$ (without the identity element).
  Let $t$ be a natural number and let $J$ be an ideal in $F$ 
such that $x_{i}^{t}\in F^{tn}+J$,  for all $i\leq n$. 
 Then there is an element $z\in F$ such that the algebra ${F\over zF+J}$ is nilpotent and hence finitely dimensional.
\end{lemma} 
\begin{proof} With a slight abuse of notation let $x_{i}$ denote the image of $x_{i}$ in $F/J$. Let $v=[x_{1}^{t}, x_{2}^{t}, \ldots , x_{n}^{t}]$. Observe that there is an $n$ by $n$ matrix $M$ with coefficients in $F/J$ such that $Mv^{T}=v^{T}$. 
   Let $I$ denote the identity matrix. By the same argument as in Theorem $75$ \cite{kap} we get that 
 $(I-M)v^{T}=0$, hence the determinant of $I-M$ annihilates all coefficients of $v$. This determinant is of the form $1+z$ with $z\in F/J$.
 It follows that $zx_{i}^{t}=x_{i}^{t}$ for all $i\leq n$. Therefore the ideal $zF+J$ contains all elements $x_{i}^{t}$ for $i\leq n$. It follows that $F/zF+J$ is a nilpotent, finitely dimensional algebra, as required. 
\end{proof}

The following lemma is well known and follows from the dimension theory.

\begin{lemma}\label{L42}
 Let $K$ be a field, and let $\bar F$ be either the free commutative algebra on the set of free generators $X=\{x_{1}, \ldots ,x_{n}\}$ or $\bar F$ be the formal commutative power series algebra over $K$ in $n$ variables $x_{1}, x_{2}, \ldots ,x_{n}$ (without the identity element).
 Let $I$ be an ideal in $\bar F$ generated by less than $n$ elements. Then ${\bar {F}\over I}$ is not a nilpotent algebra.
\end{lemma}
\begin{proof} In \cite{m} pp. $9$, section $5.4 $ we read that the dimension of a noetherian commutative ring $R$ is the minimum over all ideals $I$ of definition of $R$, of the numbers of generators of $i$. Moreover $J$ is an ideal of definition iff $R/J$ is Artinian. 
 Recall that the dimension of $\bar F$ is $n$, and $\bar F$ is noetherian. 
Hence, $\bar {F}/I$ is not Artinian, and hence it is not finitely dimensional, as required.
\end{proof}

 We are now ready to prove Theorem \ref{2}.

 \begin{proof} (Proof of Theorem \ref{2})
  Assume that $F$ is the free algebra in $n$ free variables $x_{1}, \ldots ,x_{n}$. The proof in the case when $F$ is the power series ring is the same. Recall that $x_{1}, \ldots , x_{n}$ have degree one. 
  Suppose on the contrary that $F/I$ is commutative and finitely dimensional. Then  all elements of the form 
$x_{i}x_{j}-x_{j}x_{i}$ for $i\neq j$ are in $ I$. As each $f_{i}$ is a sum of homogeneous elements of degrees  
  two or higher, it follows that each element of the form $x_{i}x_{j}-x_{j}x_{i}$ is the homogeneous component of the smallest degree  of some element $f\in I$. Without restricting the generality, we can assume that for each $i< j$, $i,j\leq n$ the element  $x_{i}x_{j}-x_{j}x_{j}$ is the lowest term of $f_{k}$ for some $k\leq {n(n-1)\over 2}$. 
 Then  $g_{k}=f_{k}-(x_{i}x_{j}-x_{j}x_{i})$  has all homogeneous components of degree $3$ or higher.
 
Let $S$ be the linear $K$-space spanned by terms of degree two in elements $f_{k}$ for $k>{n(n-1)\over 2}$, and $L$ be the linear $K$ space spanned by terms of degree two in elements $f_{k}$ for $k\leq {n(n-1)\over 2}$. Without restricting generality we can assume that $S\cap L=0$.

Observe now that the ideal $I$ equals the ideal $\bar I$ of $F$ generated by elements $f_{i}$ for $i>{n(n-1)\over 2}$ and $g_{1}, g_{2}, \ldots , g_{{n(n-1)\over 2}}$ and all elements $x_{i}x_{j}-x_{j}x_{i}$ for $i,j\leq n$. Indeed 
 as all elements $x_{i}x_{j}-x_{j}x_{i}$ are in $I$ it follows that $g_{k}\in I$ and  hence $\bar I\subseteq I$. 
On the other hand $f_{k}=g_{k}+(x_{i}x_{j}-x_{j}x_{i})\in \bar {I}$ for all $k\leq {n(n-1)\over 2}$ (and for appropriate $i,j$ depending on $k$) hence $I\subseteq \bar {I}$. It follows that $I=\bar I$.  

Let $J$ be an ideal generated by elements $x_{i}x_{j}-x_{j}x_{i}$ for $i,j\leq n$ and by elements $f_{i}$ for $i>{n(n-1)\over 2}$. 
Denote $G=Kg_{1}+Kg_{2}+\ldots +Kg_{{n(n-1)\over 2}}$.

 We claim that 
for each $j\geq 1$ we have $G\subseteq  F^{j}G+J$ . Observe that $F/J$ is commutative, hence $GF\subseteq FG+J$. We will first  show that $G\subseteq  FG+J$. 
Observe  that $g_{k}\in I$ hence $g_{k}\in \sum_{k=1}^{{n(n-1)\over 2}}\alpha _{k}f_{k} +Q$ for some $\alpha _{i}\in K$, where $Q=IF+FI+FIF +J$.

 By comparing elements of degree two in the left hand side and right hand side we get $0=\sum _{k=1}^{{n(n-1)\over 2}}\alpha _{k}t_{k}+s$ for some $s\in S$, where $t_{k}$ denotes the component of degree $2$ in $f_{k}$ for all $k\leq {n(n-1)\over 2}$. Recall that $\sum _{k=1}^{{n(n-1)\over 2}}\alpha _{k}t_{k}\in L$. It follows that  $\sum _{k=1}^{{n(n-1)\over 2}}\alpha _{k}t_{k}\in L\cap S=0$. Observe that elements $t_{k}$ are linearly independent over $K$ (because each element $t_{k}$ corresponds with exactly one element $x_{i}x_{j}-x_{j}x_{i}$ with $i<j$).  It follows that  $\alpha _{k}=0$ for all $k\leq {n(n-1)\over 2}$.

Therefore, $g_{k}\in FI+IF+FIF+J$. Observe that $I=\bar I$ hence 
$g_{k}\in F\bar {I}+\bar {I}F+F\bar {I}F+J$. Recall that $\bar {I}$ is generated by elements from $G$ and from $J$. It follows that $g_{k}\in FG+GF+FGF+J=FG+J$ as required.

 As it holds for all $k\leq d-1$ we get $G\subseteq FG+J$. By applying this observation again  we get  
 $G\subseteq  FFG+J$. Continuing in this way we get that $F^{i}G+J$ for all $j$, as required.

  We will now show that there is $m>0$ such that  $x_{1}^{m}\in F^{\alpha }+ J$ for every $\alpha >0$.
 Since $F/I$ is a finitely dimensional algebra then for some $\beta _{i}\in K$ and some numbers $m,k$ we have 
$x_{1}^{m}-\sum_{i>m}^{k} \beta _{i}x_{1}^{i}\in I$. Recall that $I=\bar {I}=KG+FG+GF+FGF+J=FG+J$.
 By our previous observation it follows that $I\subseteq F^{\alpha }G+J$.
 It follows that 
 $x_{1}^{m}\in \sum_{i>m}^{k} \beta _{i}x_{1}^{i}+F^{\alpha }G+J$. Note that 
 $\sum_{i>m}^{k} \beta _{i}x_{1}^{i}=x_{1}^{m}r$ hence $x_{1}^{m}\in x_{1}^{m}r+F^{\alpha }+J\subseteq x_{1}^{m}r^{2}+F^{\alpha }+J$. Continuing in this way (by substituting $x_{1}^{m}r+F^{\alpha }+J$ into $x_{1}^{m}$) we get $x_{1}^{m}\subseteq F^{\alpha }+J$.

Similarly, there is $t>0$ such that $x_{i}^{t}\in F^{\alpha }+J$ for all $i\leq n$ and all $\alpha >0$,  hence assumptions of Lemma \ref{L41} are satisfied. Therefore, there is $z\in F$ such that $F/J+zF$ is nilpotent. 

 Let $M$ be the ideal of $F$ generated by  all elements $x_{i}x_{j}-x_{j}x_{i}$ for $i,j\leq n$.  Recall that  $J+zF=M+T$ where $T$ is an ideal generated by elements $f_{i}$ for $i>{n(n-1)\over 2}$ and by $z$, hence by $n-1$ elements.
It follows that $\bar F=F/M$ has an ideal $Q$ generated by less than $n$ elements such that $\bar F/Q$ is nilpotent, a contradiction with Lemma \ref{L42}, thus proving the case $n>2$ of our theorem.

Let now $n=2$ and $d=2$. We know that $x_{1}^{t}+c\in J$ and  $x_{2}^{t}+c'\in J$ for some $t$ and some $c,c'\in F^{\alpha }G$. Let $\alpha >tn$. 
Observe that $J$ is generated by $x_{1}x_{2}-x_{2}x_{1}$ and $f_{2}$. Let $M$ be the ideal of $F$ generated by all elements $x_{1}x_{2}-x_{2}x_{1}$. With a slight abuse of notation let $f_{2},x_{1}, x_{2}$ denote the images of elements $f_{2}, x_{1},  x_{2}$ in  $\bar F=F/M$. Then $x_{1}^{m}+c=f_{2}h$ and 
$x_{2}^{t}+c=f_{2}g$ for some $h,g\in F$.  Let $q$ be the lowest degree term of $f_{2}$ in  $\bar F=F/M$. Then $q$ divides both $x_{1}^{m}$ and  $x_{2}^{m}$  in $\bar F$, a contradiction. 
\end{proof}
{\bf Acknowledgements.} The author is very grateful to Ivan Chelstov for providing the references for Lemma \ref{L41}, to Efim Zelmanov for his inspiring questions, and to Michael West for his continued assistance with my written English.

\begin{minipage}{1.00\linewidth}

Agata Smoktunowicz\\
 Address: Maxwell Institute for Mathematical Sciences\\
  School of Mathematics, University of Edinburgh\\
  James Clerk Maxwell Building, King's Buildings, Mayfield Road\\
  Edinburgh EH9 3JZ, Scotland, UK\\
Email: A.Smoktunowicz@ed.ac.uk\\
\end{minipage}
\end{document}